\RequirePackage[l2tabu, orthodox]{nag}

\documentclass[12pt,reqno]{amsart}
\usepackage{fullpage,url,amssymb,enumerate,colonequals,pdfsync,float,mathtools}
\usepackage{mathrsfs} % for \mathscr (script letters)

\usepackage{color}
\usepackage[dvipsnames]{xcolor}

\usepackage{tikz,tikz-cd}

\usetikzlibrary{arrows}

\usepackage{rotating}

\newcommand{\Aff}{\mathbb{A}}

\newcommand{\F}{\mathbb{F}}

\newcommand{\PP}{\mathbb{P}}
\newcommand{\Q}{\mathbb{Q}}

\newcommand{\Z}{\mathbb{Z}}

\newcommand{\calP}{\mathcal{P}}

\DeclareMathOperator{\Aut}{Aut}

\DeclareMathOperator{\Gal}{Gal}

\DeclareMathOperator{\Hom}{Hom}

\DeclareMathOperator{\Ind}{Ind}

\DeclareMathOperator{\PerFix}{PerFix}

\DeclareMathOperator{\re}{Re}

\DeclareMathOperator{\Twist}{Twist}

\newcommand{\epsh}{\varepsilon_H}

\newcommand{\IndH}{\operatorname{Ind}_H^G(1_H)}

\newcommand{\Per}{\operatorname{Per}}

\newtheorem{theorem}{Theorem}[section]
\newtheorem{lemma}[theorem]{Lemma}

\newtheorem{proposition}[theorem]{Proposition}

\theoremstyle{definition}
\newtheorem{definition}[theorem]{Definition}

\theoremstyle{remark}
\newtheorem{remark}[theorem]{Remark}

\makeatletter
\g@addto@macro\bfseries{\boldmath} % This makes math in section titles bold.
\makeatother

\usepackage{microtype}

\usepackage[
	backref,
	pdfauthor={Laura Walton}, % add other authors
]{hyperref}
\usepackage[alphabetic,backrefs,lite,nobysame]{amsrefs} % for bibliography

\begin{document}

\title[]{Counting Periodic Points over Finite Fields}
\author{Laura Walton}
\address{Mathematics Department, Brown University, Providence, RI 02912, USA}
\email{laura@math.brown.edu}

\date{\today} % eventually replace with the actual date

\begin{abstract}
Let $V$ be a quasiprojective variety defined over $\F_q$, and let $\phi:V\rightarrow V$ be an endomorphism of $V$ that is also defined over $\F_q$. Let $G$ be a finite subgroup of $\operatorname{Aut}_{\mathbb{F}_q}(V)$ with the property that $\phi$ commutes with every element of $G$. We show that idempotent relations in the group ring $\mathbb{Q}[G]$ give relations between the periodic point counts for the maps induced by $\phi$ on the quotients of $V$ by the various subgroups of $G$. We also show that if $G$ is abelian, periodic point counts for the endomorphism on $V/G$ induced by $\phi$ are related to periodic point counts on $V$ and all of its twists by $G$.
\end{abstract}

\maketitle

%****************************************************************************

\section{Introduction}

\subsection{Dynamics on varieties over finite fields.}
Let $V$ be a variety and $\phi$ an endomorphism of $V$, both defined over a field $K$. For $n\geq 1$, we define $\phi^n:=\phi\circ\ldots\circ\phi$ to be the $n$th iterate of $\phi$, and by convention $\phi^0$ is defined to be the identity map on $V$. 

\begin{definition} A point $P\in V(\overline{K})$ is called \emph{periodic} if there exists $n\geq 1$ such that $\phi^n(P)=P$, and is called \emph{preperiodic} if there exist $0\leq m <n$ such that $\phi^m(P)=\phi^n(P)$.
\end{definition}

If $K=\F_q$ for some prime power $q$, every point of $V$ is preperiodic for $\phi$, so the set of $\F_q$-preperiodic points of $V$ for $\phi$ is exactly $V(\F_q)$. While the question of bounding the number of points in $V(\F_q)$ is an interesting one, it has little to do with the dynamics of the map $\phi$. From a dynamics perspective, it is much more interesting to count the number of points in $V(\F_q)$ that are periodic for $\phi$. A growing body of literature on this topic has emerged \cites{JKM2016,Poonen2013}. Many recent results have studied the behaviour under iteration of a polynomial map $f$ on $\Aff^1(\F_q)$; see, for example \cites{Flynn-Garton2014,Ostafe-Sha2016,GXB2015,Jones-Boston2012}. There is a relative paucity of results on dynamics of general varieties over finite fields. We have give two results in this context, including a result relating periodic points of varieties and their quotients.

\begin{remark} If $K=\Q$, $V=\PP^1$ and $\phi$ is a rational map of degree at least 2, Northcott's Theorem states that $\phi$ has only finitely many preperiodic points defined over $\Q$ \cites{Northcott1950}. In the same setting, the Uniform Boundedness Conjecture of Morton and Silverman states that the number of preperiodic points is not only finite but is bounded by a constant depending only on the degree of $\phi$ \cites{Morton-Silverman1994,Morton-Silverman1995}. Our idempotent relation result (see Theorem \ref{T:secondmain}) does not hold over $\Q$, as shown by the example in Section \ref{S:example}.
\end{remark}

\subsection{Automorphisms of dynamical systems.}

As before, let $V$ be a variety defined over $\F_q$ and let $\phi:V\rightarrow V$ be an endomorphism of $V$ defined over $\F_q$. Let $G$ be a finite group of automorphisms of the dynamical system $(V,\phi)$ defined over $\F_q$; that is, $G$ is a finite subgroup of $\Aut_{\F_q}(V)$ such that $\phi$ commutes with every element of $G$. We assume that the quotient $V/G$ is defined;  in particular, this holds if $V$ is quasi-projective \cite{Mumford1970}*{Chapter 2,~Section 7}. 

\begin{definition}
We use $\Per_{\phi}(V)$ to denote the set of periodic points for $\phi$ in $V(\overline{\F}_q)$, and $\Per_{\phi}(V(\F_q))$ to denote the set of periodic points in $V(\F_q)$. We drop the subscript and write $\Per(V)$ or $\Per(V(\F_q))$ when the map is clear.
\end{definition}

\begin{definition}
Given an endomorphism $\phi$ of $V$ and a subgroup $G\leq \Aut_{\F_q}(V)$ such that $\phi$ commutes with every element of $G$, we let $\phi_G:V/G\rightarrow V/G$ to be the induced morphism on the quotient variety (See Section \ref{S:perpts} for detailed definitions). We use  $\Per_{\phi}(V/G)(\F_q)$, or simply $\Per(V/G)$, to denote the set of periodic points of $(V/G)(\F_q)$ for the map $\phi_G$.
\end{definition}

We want to compare counts of $\F_q$-periodic point counts for $\phi$ on $V$ with $\F_q$-periodic point counts for $\phi_G$ on $V/G$.

\begin{definition} Given $\chi\in H^1(\Gal(\overline{\F}_q/\F_q),G)$, we let $V^\chi$ be the twist of $V$ by $\chi$, $\phi^\chi$ the map induced by $\phi$ on $V^\chi$, and  $\Per(V^\chi)$ the periodic points on $V^\chi$ under $\phi^{\chi}$ (see Section \ref{S:twists} for precise definitions).\end{definition}

The first result we prove establishes a relationship between the number of $\F_q$-periodic points on the quotient $V/G$ and the $\F_q$-periodic points on $V$ and all of its twists.

\begin{theorem}\label{T:firstmain}
Let $V$ be a quasiprojective variety defined over $\F_q$, and let $\phi:V\rightarrow V$ be an endomorphism of $V$ that is also defined over $\F_q$. Let $G$ be a finite abelian subgroup of $\Aut_{\F_q}(V)$ with the property that $\phi$ commutes with every element of $G$.
Then, the following holds:
$$\Bigl|\Per((V/G)(\F_q))\Bigr|=\frac{1}{|G|}\left(\sum_{\chi\in H^1(\Gal(\overline{\F}_q/\F_q),G)}\Bigl|\Per(V^\chi(\F_q)) \Bigr| \right).$$
\end{theorem}

The second result establishes a relation between the quotients of $V$ by the various subgroups of $G$. If $H\leq G$ is a subgroup, we let $\epsh\in\mathbb{Q}[G]$ be its idempotent in the group ring (see Section \ref{S:idempotent} for definitions). We show the following, where the sum is taken over all subgroups $H$ of $G$.

\begin{theorem}\label{T:secondmain}
For $\alpha\in\Q[G]$, write $\alpha\sim 0$ if $\psi(\alpha)=0$ for every rational character $\psi$ of $G$. Let $V$ be a quasiprojective variety defined over $\F_q$, and let $\phi:V\rightarrow V$ be an endomorphism of $V$ that is also defined over $\F_q$. Let $G$ be a finite (but not necessarily abelian) subgroup of $\Aut_{\F_q}(V)$ with the property that $\phi$ commutes with every element of $G$.
Then, $$\sum_{H\leq G} n_H\epsh \sim 0 \Longrightarrow\sum_{H\leq G} n_H|\Per(V/H)(\F_{q})|= 0.$$
\end{theorem}

This result is a generalization of a theorem of Kani and Rosen \cite{Kani-Rosen1994}. See Section \ref{S:idempotent} for more details.

\begin{remark}
Our proof of Theorems \ref{T:firstmain} and \ref{T:secondmain} involves studying the interactions of two groups and a semigroup: the Galois group $\Gal(\overline{\F}_q/\F_q)$; the group of automorphisms $G$; and the semigroup of iterates of $\phi$, with semigroup operation given by composition.
\end{remark}

\subsection{Outline of paper.} In Section \ref{S:perpts}, we prove general lemmas that are used throughout the paper. In Section \ref{S:twists}, we show that periodic point counts on the quotient of a variety by a group $G$ are related to periodic point counts on the variety and its twists by $G$. In Section \ref{S:idempotent}, we generalize a result of Kani and Rosen on idempotent relations. We prove that idempotent relations in the group ring $\Q[G]$ imply corresponding relations amongst periodic points counts of the quotients of $V$ by the various subgroups of $G$. In Section \ref{S:example}, we provide a concrete example of a variety $V$, group of automorphisms $G$, and endomorphism $\phi$ and use this example to illustrate the idempotent result of Theorem \ref{T:secondmain} over some small finite fields. We use the same example to show that our results from Section \ref{S:idempotent} do not hold over number fields.

\section{Preliminaries}\label{S:perpts}
\subsection{Periodic Points}
Let $V$ be a variety defined over $\F_q$, and let $G$ be a finite  subgroup of $\Aut_{\F_q}(V)$. We restrict ourselves to varieties $V$ such that the quotient variety $V/G$ exists; in particular, this holds if $V$ is quasi-projective \cite{Mumford1970}*{Chapter 2,~ Section 7}. We consider the projection $\pi:V\rightarrow V/G$. Let $\phi:V\rightarrow V$ be a morphism defined over $\F_q$ such that $\phi$ commutes with every automorphism in $G$ and let $\phi_G:V/G\rightarrow V/G$ be the induced morphism on the quotient variety. The following square commmutes:
\[\begin{tikzcd}
V \arrow{r}{\phi} \arrow{d}{\pi} & V \arrow{d}{\pi} \\
V/G \arrow{r}{\phi_G}& V/G.
\end{tikzcd}
\]
As before, we use $\Per_{\phi}(V)$ to denote the set of periodic points for $\phi$. We will simply write $\Per(V)$ when the map is clear. We use  $\Per_{\phi}(V/G)$, or simply $\Per(V/G)$, to denote the set of periodic points of $V/G$ for the map $\phi_G$.

\begin{lemma}\label{perquotient}
Given $P\in V(\overline{\F}_q)$, let $\widetilde{P}:=\pi(P)$ be its image in the quotient. Then $\widetilde{P}$ is a periodic point for $\phi_G$ if and only if $P$ is a periodic point for $\phi$. 
\end{lemma}
\begin{proof}
Supposed $P$ is periodic for $\phi$ of period $n$, so $\phi^n(P)=P$. Then, since $\phi_G\circ\pi=\pi\circ\phi$, we have that $(\phi_G)^n(\widetilde{P})=(\phi_G)^n(\pi(P))=\pi(\phi^n(P))=\pi(P)=\widetilde{P}$. Thus, $\widetilde{P}$ is periodic for $\phi_G$.

Now suppose that $\widetilde{P}$ is periodic of period $n$ for $\phi_G$. It follows that $\phi^n(P)=g(P)$ for some $g\in G$. Let $m$ be the order of $g$. Then, since $\phi$ commutes with $g$, $\phi^{nm}(P)=g^m(P)=P$.
\end{proof}

An $\F_q$-periodic point for a morphism $\phi$ of $V$ is a periodic point defined over $\F_q$. We use $\Per_{\phi}(V)(\F_q)$ to denote the set of $\F_q$-periodic points for $\phi$. As before, we will simply write $\Per(V)(\F_q)$ when the map is clear. Similarly, we use  $\Per_{\phi}(V/G)(\F_q)$, or simply $\Per(V/G)(\F_q)$, to denote the set of $\F_q$-periodic points of $V/G$ for the map $\phi_G$.

\subsection{Stabilizer subgroups of preimages of quotient points.} We will use the following two lemmas several times.

\begin{lemma}\label{L:nonabpreimages}
Let $G$ be a finite subgroup of $\Aut_{\F_q}(V)$. Then, given $\widetilde{Q}\in(V/G)(\F_q)$, the stabilizer subgroups of any two preimages of $\widetilde{Q}$ are conjugate. 
\end{lemma}
\begin{proof}
Since $G$ acts transitively on the set of preimages, the stabilizer subgroups of the preimages are all conjugate.
\end{proof}

The previous lemma is used to prove the following result.

\begin{lemma}\label{L:preimages}
Let $G$ be a finite abelian subgroup of $\Aut_{\F_q}(V)$. Then, given $\widetilde{Q}\in(V/G)(\F_q)$, the stabilizer subgroups of any two preimages of $\widetilde{Q}$ are the same. 
\end{lemma}
\begin{proof}
By Lemma \ref{L:nonabpreimages}, the stabilizer subgroups of all of the preimages are all conjugate. Since $G$ is abelian, these stabilizer subgroups are in fact the same.
\end{proof}

\section{Counting periodic points with twists}\label{S:twists}
\subsection{Twists}
In order to count $\F_q$-periodic points, we begin by counting $\F_q$-points. To do so, we'll need the notion of a twist of a variety.

Given an algebraic variety $V$ defined over a perfect field $K$, a twist of $V/K$ is a variety $V'$ defined over $K$ that is $\overline{K}$-isomorphic to $V$. The set of twists of $V/K$ is the set of $K$-isomorphism classes of twists of $V/K$. We denote this set $\Twist_K(V)$.

\begin{definition}
Let $G$ be a finite abelian subgroup of $\Aut(V)$. For each $\chi\in H^1(\Gal(\overline{K}/K),G)$, there exists a variety $V^\chi$ defined over $K$ such that there exists a $\overline{K}$-isomorphism $\psi:V^\chi\rightarrow V$ so that, given $P\in V^\chi$ and $\sigma\in\Gal(\overline{K}/K)$, 
$\psi(P^\sigma)=\chi_\sigma(\psi(P)^\sigma)$ \cite{Silverman2007}*{Section 4.8}. This variety is called a $G$-\emph{twist} of $V$, and we denote the set of all $G$-twists by $\Twist_K(V, G)$.
\end{definition}

\begin{remark}
There is a map: $$\Twist_K(V, G)\rightarrow \Twist_K(V),$$
but in general this map needn't be injective.
\end{remark}

Given a $G$-twist $V^\chi$ corresponding to $\chi\in H^1(\Gal(\overline{K}/K),G)$, the $K$-points of $V^{\chi}$ correspond exactly to the points of $V$ that are invariant by the $\chi$-twisted action of the absolute Galois group $\Gal(\overline{K}/K)$; that is,
\[
V^\chi(K) = \bigl\{ Q \in V(\bar K) :
Q = \chi_\sigma(Q^\sigma)
\text{ for all }
\sigma \in Gal(\bar K/K)
\bigr\}.
\]

\begin{remark}
If $G$ is a finite \emph{abelian} subgroup of $\Aut_{\F_q}(V)$, then $$H^1(\Gal(\overline{\F}_q/\F_q),G)=\Hom(\Gal(\overline{\F}_q/\F_q),G),$$
so determining a twist is equivalent to determining a homomorphism from $\Gal(\overline{\F}_q/\F_q)$ to $G$ \cite{Silverman2007}*{Section 4.8}.
\end{remark}

We set $K=\F_q$ and will simply write $\Twist(V, G)$ for $\Twist_{\F_q}(V, G)$.

\subsection{Periodic Points and Twists}
We show that, in the setting of a finite abelian subgroup $G$ of $\Aut_{\F_q}(V)$, the $\F_q$-points of $V/G$ can be enumerated by counting $\F_q$-points of every twist in $\Twist(V,G)$.

\begin{proposition}\label{P:Stabs Correspond}
Let $G$ be an abelian subgroup of $\Aut_{\F_q}(V)$, and consider $\widetilde{Q}\in(V/G)(\F_q)$. Let $G_Q$ be the stabilizer subgroup of a preimage of $\widetilde{Q}$. Then, every preimage of $\widetilde{Q}$ corresponds to an $\F_q$-point of exactly $|G_Q|$ twists in $\Twist(V, G)$.
\end{proposition}
\begin{proof}

Since $\widetilde{Q}\in (V/G)(\F_q)$, it follows that  $\sigma({\widetilde{Q}})=\widetilde{Q}$ for every $\sigma\in \Gal(\overline{\F}_q/\F_q).$ Let $\pi^{-1}(\widetilde{Q})=\{Q_1,\ldots, Q_m\}$ be the set of preimages of $\widetilde{Q}$ in $V(\overline{\F}_q)$. By Lemma \ref{L:preimages}, the stabilizer subgroups of the preimages are all the same. Let $G_Q$ denote this stabilizer subgroup.

By the transitivity of $G$, it is enough to prove that $Q_1$ corresponds to an $\F_q$-point of exactly $|G_Q|$ twists of $V$.  Let $\sigma_q$ be a topological generator of $\Gal(\overline{\F}_q/\F_q)$. Since $\widetilde{Q}$ is Galois-invariant, the action of $\Gal(\overline{\F_q}/\F_q)$ on $\pi^{-1}(\widetilde{Q})$ permutes the preimages. Thus, there exists $g_q\in G$ so that $Q_1^{\sigma_q}=g_q(Q_1)$. It follows that $Q_1=g_q^{-1}\left(Q_1^{\sigma_q}\right)$.

Let $\chi_q\in H^1(\Gal(\overline{\F}_q/\F_q),G)=\Hom(\Gal(\overline{\F_q}/\F_q),G)$ 
be the homomorphism determined by $\sigma_q\mapsto \chi_{q,\sigma_q}=g_q^{-1}$. 
The $\F_q$-points of $V^{\chi_q}$ correspond exactly to the points 
$Q\in V$ such that $Q=\chi_{q,\sigma}\left(Q^\sigma\right)$ for all 
$\sigma\in \Gal(\overline{\F}_q/\F_q)$. 
Since $\sigma_q$ is a topological generator for 
$\Gal(\overline{\F}_q/\F_q)$,  to find the points $Q\in V$ corresponding to $\F_q$-points in $V^{\chi_q}$,
it is enough to check whether $Q=\chi_{q,\sigma_q}\left(Q^{\sigma_q}\right)$.

In the case of $Q_1$, we see that $\chi_{q,\sigma_q}\left(Q_1^{\sigma_q}\right)=g_q^{-1}\left(Q_1^{\sigma_q}\right)=Q_1$. Thus, $Q_1$ corresponds to an $\F_q$-point in $V^{\chi_q}$. The homomorphism $\chi_q$ depends on the choice of an element $g_q\in G$ such that $Q_1^{\sigma_q}=h_q(Q_1)$, and the choice of any other group element would yield a different homomorphism corresponding to a different twist of $V$. We wish to count how many such elements there are. Given an element in the stabilizer subgroup $s\in G_Q$, the element $sg_q\in G$ also has the property that $Q_1^{\sigma_q}=sg_q(Q_1)$. Moreover, if $h_q\in G$ has the property that $Q_1^{\sigma_q}=h_q(Q_1)$, then $g_q^{-1}h_q(Q_1)=Q_1$, so $h_q$ differs from $g_q$ by an element of the stabilizer $G_Q$. Thus, we conclude that there are exactly $|G_Q|$ choices for $g_q$, and therefore $|G_Q|$ twists on which $Q_1$ corresponds to an $\F_q$-point.
\end{proof}

The following result is an easy consequence of Proposition \ref{P:Stabs Correspond}.

\begin{proposition} \label{P:Fq Points Correspond}
Let $G$ be an abelian subgroup of $\Aut_{\F_q}(V)$, and consider $\widetilde{Q}\in(V/G)(\F_q)$. Then, there are exactly $|G|$ twists in $\Twist(V, G)$ on which some preimage of $\widetilde{Q}$ is an $\F_q$ point.
\end{proposition}
\begin{proof}
We notice that the number of preimages of $\widetilde{Q}$ is $m=|G|/|G_Q|$. By Proposition \ref{P:Stabs Correspond}, each preimage is an $\F_q$ point on exactly $|G_Q|$ twists. 
\end{proof}

From this, we get a result on point counting.

\begin{proposition}
Let $G$ be a finite abelian subgroup of $\Aut_{\F_q}(V)$. Then:

$$|(V/G)(\F_q)|=\frac{1}{|G|}\left(\sum_{\chi\in H^1(\Gal(\overline{\F}_q/\F_q),G)}|V^{\chi}(\F_q)  | \right).$$
\end{proposition}
\begin{proof}
The result follows from Proposition \ref{P:Fq Points Correspond} and counting points.
\end{proof}

Given a twist $V^{\chi}\in \Twist_K(V,G)$, let $\psi_\chi: V\rightarrow V^{\chi}$ be a $\overline{K}$-isomorphism. For any morphism $\phi:V\rightarrow V$, we define the twisted morphism $\phi_{\chi}:V^{\chi}\rightarrow V^{\chi}$ to be the morphism making the following square commute, i.e. $\phi_\chi:=\psi_\chi\circ\phi\circ\psi_\chi^{-1}$:
\[\begin{tikzcd}
V \arrow{r}{\phi} \arrow{d}{\psi_\chi} & V \arrow{d}{\psi_\chi} \\
V^{\chi} \arrow{r}{\phi_\chi}& V^\chi.
\end{tikzcd}
\]
We claim that periodic points of $V$ under $\phi$ correspond to periodic points of $V^\chi$ under $\phi_\chi$.

\begin{proposition}\label{P:Per Points Correspond}
Given $V^{\chi}\in \Twist_K(V,G)$, a point $Q\in V$ is periodic for $\phi$ if and only if the point $P:=\psi_\chi(Q)\in V^{\chi}$ is periodic for $\phi^{\chi}$.
\end{proposition}

\begin{proof}
This follows from the commutativity of the diagram above. If $Q\in V$ is periodic for $\phi$, then $\phi^n(Q)=Q$ for some $Q$. Let $P=\psi_\chi(Q)$. Then $$\left(\phi_\chi\right)^n(P)=\left(\phi_\chi\right)^n\left(\psi_{\chi}(Q)\right)=\psi_\chi\left(\phi^n(Q)\right)=\psi_\chi(Q)=P,$$
so $P$ is periodic for $\phi_\chi$.

Now, suppose $P$ is periodic for $\phi^{\chi}$, so $\left(\phi_\chi\right)^n(P)=P$ for some $n$. Since $\psi^{\chi}$ is an isomorphism, we have that $\phi\circ \psi_\chi^{-1}=\psi_\chi^{-1}\circ \phi_\chi$. Then, 

$$\phi^n(Q)=\phi^n\left(\psi_{\chi}^{-1}(P)\right)=\psi_\chi^{-1}\left(\phi_{\chi}^n(P)\right)=\psi_\chi^{-1}(P)=Q,$$
so $Q$ is periodic for $\phi$.
\end{proof}

In fact, with a little more care, the argument above shows that a point in $V$ is periodic of period $n$ if and only if the corresponding point in $V^\chi$ is periodic of period $n$. However, the previous lemma is all that is needed to prove the following claim.

\begin{proposition}\label{P:twistcorresp}
Let $G$ be a finite abelian subgroup of $\Aut_{\overline{\F}_q}(V)$ on which $\Gal(\overline{\F}_q/\F_q)$ acts trivially, i.e. $G$ is an abelian subgroup of $\Aut_{\F_q}(V)$. Then every periodic point in $\Per((V/G)(\F_q))$ corresponds to an $\F_q$-periodic point on exactly $|G|$ twists in $\Twist(V, G)$.
\end{proposition}
\begin{proof}
By Proposition \ref{P:Fq Points Correspond}, every point in $\Per((V/G)(\F_q))$ corresponds to $|G|$ $\F_q$-points on the various twists $V^\chi$. The result is a direct consequence of this fact and Proposition \ref{P:Per Points Correspond}.
\end{proof}

We will now prove the first main result of the paper.

\begin{proof}[Proof of Theorem \ref{T:firstmain}]
This follows from Proposition \ref{P:twistcorresp} by counting points.
\end{proof}
\section{Idempotent relations of periodic points}\label{S:idempotent}
\subsection{Background on idempotent relations}
Let $G$ be a finite group. To each subgroup $H\leq G$ we associate the idempotent 
$$\epsh=|H|^{-1}\sum_{h\in H} h \in \mathbb{Q}[G].$$ 

\begin{definition} A relation of the form $$\sum_{H\subset G}n_H \epsh=0,$$ with $n_H\in \mathbb{Z}$, is called an 
\emph{idempotent relation}.

Given $\alpha\in \Q[G]$, we 
write $\alpha\sim0$ if $\psi(\alpha)=0$ for every rational 
character $\psi$ of $G$.
\end{definition}

\begin{remark}\label{R:partitionrelation}  Every non-cyclic group $G$ admits a nontrivial idempotent relation. A special class of idempotent relations are induced by \emph{partitions} of a group $G$, by which we mean decompositions $G=\bigcup_{i=1}^k H_i$, where $H_1,\ldots,H_k$ are subgroups of $G$ such that $H_i\cap H_j=\{\operatorname{id}\}$ for $i\neq j$. From such a partition we get the idempotent relation:

$$|G|\varepsilon_G +(k-1)\varepsilon_{\operatorname{id}}=\sum_{i=1}^k |H_i|\varepsilon_{H_i}.$$
\end{remark}

Kani and Rosen prove the following theorem:

\begin{theorem}[Kani-Rosen \cite{Kani-Rosen1994}*{Proposition 3.1}]\label{T:point count relation}
Let $V$ be a quasiprojective variety over $\mathbb{F}_q$. Suppose $G$ is a finite subgroup of $\Aut_{\F_q}(V)$. Then,

$$\sum_{H\subset G} n_H\epsh\sim0 \implies \sum_{H\subset G} n_H |(V/H)(\F_{q})|=0.$$
\end{theorem}

We prove that the same is true of the periodic point counts. In \cites{Kani-Rosen1994}, Kani and Rosen prove the following result.

\begin{lemma}[Kani-Rosen \cite{Kani-Rosen1994}*{Proposition 1.1}]\label{L:character relations} Let $G$ be a finite subgroup of $\Aut_{\F_q}(V)$, and let $H$ be a subgroup of $G$. Let $\psi_H=\IndH$ be the character on $G$ induced by the trivial character on $H$. Then,
$$\sum_{h\in H} n_H \psi_H=0 \iff \sum_{h\in H} n_H \epsh \sim0.$$
\end{lemma}

For the convenience of the reader, we reproduce the proof here.

\begin{proof}
Given characters $\alpha$ and $\beta$ on $G$, recall the pairing:
$$( \alpha ,\beta )_G :={\frac {1}{|G|}}\sum _{g\in G}\alpha (g){\overline {\beta (g)}}.$$

Using Frobenius reciprocity, we see that:

\begin{equation}
\label{E:frob recip}
\psi(\epsh)=|H|^{-1}\sum_{h\in H} \psi(h)=(\psi|_H,1_H)_H=(\psi,\psi_H)_G.
\end{equation}

First, suppose that $\sum n_H \psi_H=0$. Then, using \eqref{E:frob recip}, for any rational character $\psi$ we have that
$$\psi\left(\sum_{h\in H} n_H \epsh\right)=\sum_{h\in H} n_H \psi(\epsh)= \sum_{h\in H} n_H (\psi,\psi_H)_G=  \left(\psi,\sum_{h\in H} n_H\psi_H\right)_G=0,$$
and thus $\sum n_H \epsh \sim0$.

Now, suppose that $\sum n_H \epsh \sim0$. For any rational character $\psi$, $\psi(\sum n_H \epsh)=0$. In particular, this holds for $\psi=\sum n_H \psi_H$. Setting $\psi=\sum n_H\psi_H$, we find using \eqref{E:frob recip} that
$$0=\psi\left(\sum_{h\in H} n_H \epsh\right)=\left(\psi,\sum_{h\in H} n_H\psi_H\right)_G=\left(\sum_{h\in H} n_H\psi_H,\sum_{h\in H} n_H\psi_H\right)_G,$$ so we conclude that $\sum n_H\psi_H=0$.
\end{proof}

\subsection{Idempotent relations and periodic points}

In \cite{Kani-Rosen1994}, Kani and Rosen prove their results using the $\zeta$- and $L$-functions described in \cite{Serre1965}. Generalizing their work, we introduce periodic analogues of these $\zeta$- and  $L$-functions and detail some of their properties.

\begin{definition}The \emph{periodic zeta function} for $\phi$ is the function:

$$\zeta_\phi(V,s):=\exp\left(\sum_{n=1}^\infty |\Per_\phi(V)(\F_{q^n})|\frac{q^{-ns}}{n}\right).$$
\end{definition}
We write $\zeta_\phi(V/G,s)$ to denote the function $\zeta_{\phi_G}(V/G,s)$. It is convenient to make the substitution $u=q^{-s}$ and define 
$$Z_\phi(V,u):=\exp\left(\sum_{n=1}^\infty |\Per_\phi(V)(\F_{q^n})|\frac{u^n}{n}\right);$$
 as before we write $Z_\phi(V/G, u)$ to denote $Z_{\phi_G}(V/G, u).$

Let $\sigma_q:V\rightarrow V$ be the Frobenius morphism of $V$ onto itself. The set of fixed points of the action of $\sigma_q^n$ on $V(\bar{\F}_q)$ is $V(\F_{q^n})$. 

\begin{definition}
Let $\psi$ be a character of $G$. Let $|\PerFix(g\sigma_q^n)|$ count the number of fixed points of $g\sigma_q^n$ acting on $\Per(V)(\bar{\F}_q)$, and define $$\nu_n(\psi)=\frac{1}{|G|}\sum_{g\in G} \psi(g^{-1})|\PerFix(g\sigma_q^n)|.$$
The \emph{periodic $L$ function for $\phi$} is the function $L_\phi(V,\psi,s)$ defined by
$$ L_\phi(V,\psi,s)=\exp\left(\sum_{n=1}^\infty \nu_n(\psi)\frac{q^{-ns}}{n}\right).$$\end{definition}

We will need the following properties of the periodic $L$ function.

\begin{lemma}\label{L:properties}
$L_\phi(V,\psi,s)$ has the following properties.
\begin{enumerate}[(a)]

\item $L_\phi(V,\psi,s)$ is analytic in $\re(s)>\dim(V)$.
\item $L_\phi(V,\psi+\psi',s)=L_\phi(V,\psi,s)L_\phi(V,\psi',s)$.
\item $L_\phi(V,0,s)=1$.
\item If $\psi_H=\Ind_H^{G}(1_H)$, then
$$L_\phi(V,\psi_H,s)=\zeta_\phi(V/H, s).$$
\end{enumerate}
\end{lemma}

\begin{proof}
\begin{enumerate}[(a)]
\item This follows from \cites{Serre1965}.
\item Clear since $\nu_n(\psi+\psi')=\nu_n(\psi)+\nu_n(\psi')$.
\item Clear since $\nu_n(0)=0$.
\item Notice that
\begin{align*}\nu_n(\psi_H)&=\frac{1}{|G|}\sum_{g\in G} \psi_H(g^{-1})|\PerFix(g\sigma_q^n)|\\
&=\frac{1}{|G|}\sum_{h\in H} \frac{|G|}{|H|}|\PerFix(h\sigma_q^n)|\\
&=\frac{1}{|H|}\sum_{h\in H} |\PerFix(h\sigma_q^n)|.\end{align*} 
It remains to show that 
$$\frac{1}{|H|}\sum_{h\in H} |\PerFix(h\sigma_q^n)|=|\Per(V/H)(\mathbb{F}_{q^n})|,$$
whence the claim follows.

First, suppose that $Q\in \PerFix(h\sigma_q^n)$. Let $\widetilde{Q}$ be the image of $Q$ in the quotient $V/H$. By Lemma \ref{perquotient},  $\widetilde{Q}$ is a periodic point for $\phi_H$. Since $Q$ is fixed under $h\sigma_q^n$, it follows that $Q^{\sigma_q^n}=h^{-1}(Q),$ so $\widetilde{Q}^{\sigma_q^n}=\widetilde{Q}$. Thus, any element $Q\in \PerFix(h\sigma_q^n)$ is the preimage of a point in $\Per(V/H)(\mathbb{F}_{q^n})$

Given $\widetilde{Q}\in\Per(V/H)(\mathbb{F}_{q^n})$, let $\pi^{-1}(\widetilde{Q})=\{Q_1,\ldots, Q_m\}$ be the set of preimages of $\widetilde{Q}$ in $V$. By Lemma \ref{perquotient}, all of these preimages are periodic points.

Since $\widetilde{Q}^{\sigma_q^n}=\widetilde{Q}$, the action of $\sigma_q^n$ on $\pi^{-1}(\widetilde{Q})$ permutes the preimages. Thus, there exists $h_q\in H$ so that $Q_1^{\sigma_q^n}=h_q(Q_1)$. It follows that 
$$Q_1=h_q^{-1} \left(Q_1^{\sigma_q^n}\right)=h_q^{-1}\circ\sigma_q^n\left(Q_1\right),$$ therefore $Q_1\in \PerFix(h_q^{-1}\sigma_q^n)$. By Lemma \ref{L:nonabpreimages}, the stabilizer subgroups of $h$ of the preimages of $\widetilde{Q}$ are all conjugate and therefore all have the same order. Let $|H_Q|$ denote the order of the stabilizer subgroups of the preimages, and let $H_{Q_1}$ be the stabilizer subgroup of $Q_1$.

We see that $Q_1\in \PerFix(h^{-1}\sigma_q^n)$ provided the element $h\in H$ has the property that $Q_1^{\sigma_q^n}=h(Q_1)$. We found that $h_q$ has this property and wish to count how many such elements there are. Given an element in the stabilizer subgroup $s\in H_{Q_1}$, the element $sh_q\in G$ also has the property that $Q_1^{\sigma_q}=sh_q(Q_1)$. Moreover, if $k_q\in H$ has the property that $Q_1^{\sigma_q}=k_q(Q_1)$, then $h_q^{-1}k_q(Q_1)=Q_1$, so $k_q$ differs from $k_q$ by an element of the stabilizer $H_{Q_1}$. Thus, we conclude that there are exactly $|H_{Q_1}|=|H_{Q}|$ choices for $h_q$, and therefore there are exactly $|H_{Q}|$ sets of the form $\PerFix(h\sigma_q^n)$ that contain $Q_1$. The same is true for every other preimage $Q_i$ of $\widetilde{Q}$. There are exactly $|H|/|H_Q|$ preimages, so $\widetilde{Q}$ corresponds to $|H|$ points in the various sets of the form $\PerFix(h\sigma_q^n)$ (counted with multiplicity).

\end{enumerate}
\end{proof}

These properties will help prove the following result.

\begin{proposition} \label{P:sumtoproduct}
With notation as before, $$\sum_{H\subset G} n_H\epsh \sim 0\Longrightarrow\prod_{H\subset G} \zeta_\phi(V/H, s)^{n_H}=1.$$
\end{proposition}

\begin{proof}
By Lemma \ref{L:character relations}, $\sum_{H\subset G} n_H\epsh \sim 0$ if and only if $\sum_{H\subset G} n_H\psi_H= 0.$ Using properties (b), (c), and (d) from Lemma \ref{L:properties}, we see that:
\begin{align*}1&=L_\phi(V,0,s)\\&=L_\phi\left(V,\sum_{H\subset G} n_H\psi_H, s\right)\\
&=\prod_{H\subset G} L_\phi(V,\psi_H, s)^{n_H}\\
&=  \prod_{H\subset G}\zeta_\phi(V/H, s)^{n_H}.\end{align*}
\end{proof}

Using the previous result, we prove Theorem \ref{T:secondmain}, which says that idempotent relations imply relations on periodic point counts. 

\begin{proof}[Proof of Theorem \ref{T:secondmain}]
By Proposition \ref{P:sumtoproduct}, $\sum_{H\subset G} n_H\epsh \sim 0$ implies that $\prod_{H\subset G} \zeta_\phi(V/H, s)^{n_H}=1$. Making the substitution $u=q^{-s}$, we see that 
$\prod_{H\subset G} Z_\phi(V/H, u)^{n_H}=1$. Taking the logarithm of both sides and equating coefficients of $u^n$, the result follows.
\end{proof}
\begin{remark}
By letting $\phi$ be the identity morphism on $V$, the result of Theorem \ref{T:secondmain} is exactly the result from Kani-Rosen \cite{Kani-Rosen1994}*{Proposition 3.1} cited in Theorem \ref{T:point count relation}.
\end{remark}

\section{Example}\label{S:example}

\subsection{Idempotent relations for a rational map of projective space.} In this section, we give an example to illustrate the periodic point count relation induced by an idempotent relation on projective space of dimension one over certain small finite fields. We also show that the idempotent relation does not generally induce a periodic point count relation over number fields by showing that the induced relation on periodic points does not hold for the same example with base field $K=\Q$.

Let $V:=\mathbb{P}^1_x.$ The maps $$\sigma:\mathbb{P}^1_x\rightarrow\mathbb{P}^1_x \text{ and } \tau:\mathbb{P}^1_x\rightarrow\mathbb{P}^1_x$$ given by $$\sigma(x)=-x\text{ and }\tau(x)=\frac{1}{x}$$
are commuting automorphisms of $\mathbb{P}^1_x\rightarrow\mathbb{P}^1_x$. Let $G:=\left<\sigma,\tau\right>$ be the group generated by these automorphisms under composition. Then $G$ is an abelian group, and $G\cong (\Z/2\Z)^2$. The set of subgroups of $G$ is given by:
$$\left\{G, H_\sigma:=\left<\sigma\right>,H_\tau:=\left<\tau\right> , H_{\sigma\tau}:=\left<\sigma\tau\right>, H_{\operatorname{id}}:=\left<\operatorname{id}\right> \right\},$$
so we have the diagram of subgroup containments below.

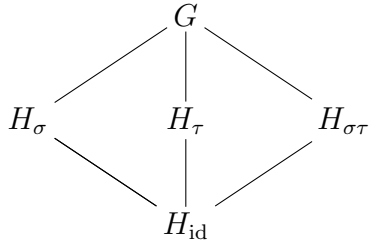
\begin{center}
\begin{tikzpicture}[scale=.7]
  \node (one) at (0,2) {$G$};
  \node (a) at (-3,0) {$H_\sigma$};
  \node (b) at (0,0) {$H_\tau$};
  \node (c) at (3,0) {$H_{\sigma\tau}$};
  \node (zero) at (0,-2) {$H_{\operatorname{id}}$};
  \draw  [black, shorten <=-2pt, shorten >=-2pt] (one) to (a);
  \draw  [black, shorten <=-2pt, shorten >=-2pt] (one) to (b);
    \draw  [black, shorten <=-2pt, shorten >=-2pt] (one) to (c);
      \draw [black, shorten <=-2pt, shorten >=-2pt] (a) to (zero);
        \draw  [black, shorten <=-2pt, shorten >=-2pt] (a) to (zero);  \draw [black, shorten <=-2pt, shorten >=-2pt] (b) to (zero);
          \draw  [black, shorten <=-2pt, shorten >=-2pt] (c) to (zero);
\end{tikzpicture}
\end{center}

Let $\varepsilon_G,\varepsilon_\sigma,\varepsilon_{\sigma\tau},\varepsilon_\tau,\text{ and }\varepsilon_{\operatorname{id}}$ be the idempotents in the group ring $\mathbb{Q}[G]$ associated to the subgroups 
$G,H_\sigma, H_{\sigma\tau},H_\tau,\text{ and } H_{\operatorname{id}},$ respectively. Notice that we have a partition of $G$ given by $G=H_\sigma\cup H_{\sigma\tau}\cup H_\tau$. Following Remark \ref{R:partitionrelation}, there is an idempotent relation given by:

$$2\varepsilon_G-\varepsilon_\sigma-\varepsilon_{\sigma\tau}-\varepsilon_\tau+\varepsilon_{\operatorname{id}}=0.$$

 The endomorphism $\phi:\mathbb{P}^1_x\rightarrow\mathbb{P}^1_x$ given by $$\phi(x)=\frac{x(x^2+2)}{2x^2+1}$$
commutes with the automorphisms $\sigma\text{ and }\tau$. For this reason, it is natural to study the maps induced by $\phi$ on the quotients of $\PP^1$ by the various subgroups of $G$.

Since all of the quotients of $\PP^1$ by the subgroups of $G$ are isomorphic to $\PP^1$, we introduce the notation 
$$\PP^1_\sigma:=\PP^1 /{H_\sigma}\cong\PP^1,~~~\PP^1_\tau:=\PP^1 /{H_\tau}\cong\PP^1,~~~\PP^1_{\sigma\tau}:=\PP^1 /{H_{\sigma\tau}}\cong\PP^1,~~~\PP^1_G:=\PP^1 /{G}\cong\PP^1.$$ We define the projections

\begin{align*}
 \pi_\sigma:\PP^1&\rightarrow \PP^1_\sigma & \pi_{\sigma\tau}:\PP^1&\rightarrow \PP^1_{\sigma\tau} & \pi_{\tau}:\PP^1&\rightarrow \PP^1_\tau & \pi_{G}:\PP^1&\rightarrow \PP^1_G  \\
 x&\mapsto u:=x^2, & x&\mapsto v:=x-\frac{1}{x}, & x&\mapsto w:=x+\frac{1}{x},& x&\mapsto z:=x^2+\frac{1}{x^2}, \\
\end{align*}
and further define the projections $\nu_1:\PP^1_\sigma\rightarrow \PP^1_G$, $\nu_2:\PP^1_{\sigma\tau}\rightarrow \PP^1_G$, and $\nu_3:\PP^1_\tau\rightarrow \PP^1_G$ so that the following diagram commutes:
\begin{center}
\begin{tikzcd}[column sep=1.4cm]
{} & \mathbb{P}^1_\sigma\arrow{dr}{\pi_\tau} \arrow{d}{\pi_{\sigma\tau}} \arrow{dl}[swap]{\pi_\sigma} & \\    \mathbb{P}^1_\sigma\arrow{dr}[swap]{\nu_1} &\mathbb{P}^1_{\sigma\tau}\arrow{d}{\nu_2} & \mathbb{P}^1_\tau\arrow{dl}{\nu_3} \\    & \mathbb{P}^1_G&
\end{tikzcd}
\end{center}
and $\nu_1\circ\pi_\sigma=\nu_2\circ\pi_{\sigma\tau}=\nu_1\circ\pi_\tau=\pi_G$.
Now, we wish to study the maps induced by $\phi$ on the quotients. We define $\phi_\sigma$ to be the map ensuring the commutativity of the following diagram
\[\begin{tikzcd}
\PP_1 \arrow{r}{\phi} \arrow{d}{\pi_\sigma} & \PP_1 \arrow{d}{\pi_\sigma} \\
\mathbb{P}^1_\sigma \arrow{r}{\phi_\sigma}& \mathbb{P}^1_\sigma,
\end{tikzcd}
\]
and we define $\phi_{\sigma\tau}$, $\phi_\tau$ and $\phi_G$ analogously. Computations give that:

$$\phi_\sigma(u)=\frac{u(u+2)^2}{(2u+1)^2},~\phi_{\sigma\tau}(v)=\frac{v^3+3v}{2v^2+9},~\phi_{\tau}(w)=\frac{w^3+5w}{2w^2+1}, ~\text{and } \phi_G(z)=\frac{z^3+8z^2+37z+48}{(2z+5)^2}. $$

For the sake of space, we introduce some notation. Given a subgroup $H$ of $G$ and a field $K$, we write $\calP_H(K)$ to denote $\Per_{\phi_H}(\PP^1/H)(K)$. So, for example, $\calP_{H_\sigma}(\Q)= \Per_{\sigma}(\PP^1_\sigma)(\Q)$  We used Sage to find the following sets and counts of $K$-periodic points \cite{sage}.

\begin{table}[H]
    \centering
    \begin{tabular}{|c|c|c|c|c|c|c|}
    \hline
    $H\leq G$ & $\calP_H(\Q)$ & $|\calP_H(\Q)|$  &$\calP_H(\F_5)$ & $|\calP_H(\F_5)|$ & $\calP_H(\F_7)$ &$|\calP_H(\F_7)|$\\\hline
    
    $H_{\operatorname{id}}$ &$\{0,\pm1,\infty\}$ &4 &$\{0,1,2,3,4,\infty\}$ &6 &$\{0,1,6,\infty\}$ &4\\    \hline 
    
    $ H_\sigma$ & $\{0,\pm1,\infty\}$ &4 &$\{0,1,4,\infty\}$ &4 &$\{0,1,6,\infty\}$ &4\\ \hline
    
    $H_{\sigma\tau}$ &$\{0,\infty\}$ &2 &$\{0,1,2,3,4,\infty\}$  &6 &$\{0,1,6,\infty\}$  &4\\ \hline
    
    $H_\tau$ &$\{0,\pm 2,\infty\}$ &4 &$\{0,2,3,\infty\}$ &4 &$\{0,2,5,\infty\}$ &4\\ \hline
    
     $G$ & $\{-4,\pm2,\infty \}$ &4 &$\{1,2,3,\infty\}$ &4 &$\{2,3,5,\infty\}$ &4\\ \hline

    \end{tabular}
    \caption{Sets and counts of periodic points in $\PP^1/H$ for the fields $\Q$, $\F_5$, and $\F_7$}
    \label{tab:example point counts}
\end{table}

Recall that we have the idempotent relation:
$$2\varepsilon_G-\varepsilon_\sigma-\varepsilon_{\sigma\tau}-\varepsilon_\tau+\varepsilon_{\operatorname{id}}=0.$$

In the case of $K=\F_5$, we see that we have the relation:

\begin{align*} 2|\calP_G(\F_5)|-|\calP_{\sigma}(\F_5)|- |\calP_{{\sigma\tau}}(\F_5)|-|\calP_{\tau}(\F_5)|+|\calP_{{\operatorname{id}}}(\F_5)|&=2\cdot4-4-6-4+6=0.\end{align*}

In the case of $K=\F_7$, the relation induced by the idempotent relation also holds:

\begin{align*} 2|\calP_G(\F_7)|-|\calP_{\sigma}(\F_7)|- |\calP_{{\sigma\tau}}(\F_7)|-|\calP_{\tau}(\F_7)|+|\calP_{{\operatorname{id}}}(\F_7)|&=2\cdot4-4-4-4+4=0.\end{align*}

However, we do not get a corresponding relation on the $\Q$-periodic points, as can be seen below:

\begin{align*} 2|\calP_G(\Q)|-|\calP_{\sigma}(\Q)|- |\calP_{{\sigma\tau}}(\Q)|-|\calP_{\tau}(\Q)|+|\calP_{{\operatorname{id}}}(\Q)|&=2\cdot4-4-2-4+4=2\neq0.\end{align*}

Thus, we see that idempotent relations do not always induce periodic point count relations over number fields.

%%%%%%%%%%%%%%%%%%%%%%%%%%%%%%%%%%%%%%%%%%%%%%%%%

\section*{Acknowledgements} 
The author would like to thank her advisor, Joseph Silverman, for suggesting this project, for many helpful conversations, and for his guidance in the process of preparing this paper.
The author would also like to thank Robert L. Benedetto and Jamie Juul for helpful conversations during a visit to Amherst College.

\bibliographystyle{plain}
\bibliography{bib.bib}

\end{document}